\documentclass{llncs}
\usepackage{amsmath}
\usepackage{amssymb}
\usepackage{enumerate}
\usepackage{varioref}
\usepackage{charter,eulervm}
\usepackage{graphicx}

\newcommand{\es}{\varnothing}

\title{{\sc Independent Sets in Edge-Clique Graphs}}

\author{
 Maw-Shang~Chang\inst{1}
\and 
 Ton~Kloks\inst{2} 
\and
 Ching-Hao~Liu\inst{2}} 
\institute{
 Department of Computer Science and Information Engineering\\
 Hungkuang University, Taiwan\\
 {\tt mschang@sunrise.hk.edu.tw}
\and 
 Department of Computer Science\\
 National Tsing Hua University, Taiwan\\
 {\tt chinghao.liu@gmail.com}
}

\begin{document}

\maketitle

\begin{abstract}
We show that the edge-clique graphs of cocktail party graphs 
have unbounded rankwidth. This, and other observations lead us 
to conjecture that the edge-clique cover problem is NP-complete 
for cographs.  
We show that the independent set problem 
on edge-clique graphs of cographs and of distance-hereditary 
graphs can be solved in $O(n^4)$ time. We show that the independent 
set problem on edge-clique graphs of graphs without odd wheels 
remains NP-complete.  
\end{abstract}

\section{Introduction}

Let $G=(V,E)$ be an undirected graph with vertex set $V$ and 
edge set $E$. A clique is a complete subgraph of $G$. 

\begin{definition}
An edge-clique covering of $G$ is a family of 
complete subgraphs such that 
each edge of $G$ is in at least one member of the family. 
\end{definition}
The minimal cardinality of such a family is the 
edge-clique covering number, and we denote it by 
$\theta_e(G)$. 

\bigskip 

The problem of deciding if $\theta_e(G) \leq k$, 
for a given natural number $k$, 
is NP-complete~\cite{kn:kou,kn:orlin,kn:holyer}. 
The problem remains NP-complete when restricted to graphs 
with maximum degree at most six~\cite{kn:hoover}.  
Hoover~\cite{kn:hoover} gives a polynomial time algorithm for 
graphs with maximum degree at most five. 
For graphs with 
maximum degree less than five, this was already done 
by Pullman~\cite{kn:pullman}. 
Also for linegraphs the problem can be solved in 
polynomial time~\cite{kn:orlin,kn:pullman}. 
 
In~\cite{kn:kou} it is shown that approximating the 
clique covering number within a constant factor 
smaller than two remains 
NP-complete. 

\bigskip 

Gy\'arf\'as~\cite{kn:gyarfas} showed the following 
interesting lowerbound. 
Two vertices $x$ and $y$ are {\em equivalent\/} if 
they are adjacent and have the same closed neighborhood. 

\begin{theorem}
\label{gyarfas}
If a graph $G$ has $n$ vertices and contains neither isolated 
nor equivalent vertices then $\theta_e(G) \geq \log_2(n+1)$. 
\end{theorem}

Gy\'arf\'as result implies that the edge-clique 
cover problem is fixed-parameter 
tractable (see also~\cite{kn:gramm}). 
Cygan et al showed that, under the assumption of the 
exponential time hypothesis, there is 
no polynomial-time algorithm which reduces the 
parameterized problem $(\theta_e(G),k)$ to a kernel of size bounded 
by $2^{o(k)}$. In their proof the authors make use of the fact that 
$\theta_e(cp(2^{\ell}))$ is a [sic] 
``hard instance for the edge-clique cover problem, at least from a 
point of view of the currently known algorithms.''    
Note that, in contrast, the parameterized edge-clique partition 
problem can be reduced to a kernel with at most $k^2$ 
vertices~\cite{kn:mujuni}. (Mujuni and Rosamond also mention that 
the edge-clique cover problem probably has no polynomial kernel.)   

\section{Rankwidth of edge-clique graphs of cocktail parties}

\begin{definition}
The cocktail party graph $cp(n)$ is the complement of a 
matching with $2n$ vertices. 
\end{definition}

Notice that a cocktail party graph has no 
equivalent vertices. Thus, by Theorem~\ref{gyarfas}, 
\[\theta_e(cp(n)) \geq log_2(2n+1).\] 
For the cocktail party graph 
an exact formula for $\theta_e(cp(n))$ 
appears in~\cite{kn:gregory}. 
In that paper Gregory and Pullman prove that 
\[\lim_{n \rightarrow \infty} \; \frac{\theta_e(cp(n))}{\log_2(n)}=1.\] 

\bigskip 
   
\begin{definition}
Let $G=(V,E)$ be a graph. The edge-clique graph $K_e(G)$ has 
as its vertices the edges of $G$ and two vertices of 
$K_e(G)$ are adjacent when the corresponding edges in $G$ are 
contained in a clique. 
\end{definition}

Albertson and Collins prove that there is a 1-1 correspondence 
between the maximal cliques in $G$ and $K_e(G)$~\cite{kn:albertson}. 
The same holds true 
for the intersections of maximal cliques in $G$ and in $K_e(G)$. 

\bigskip 

For a graph $G$ we denote the vertex-clique cover number of 
$G$ by $\kappa(G)$. 
Thus 
\[\kappa(G) = \chi(\Bar{G}).\]  
Notice that, for a graph $G$,  
\[\theta_e(G)=\kappa(K_e(G)).\] 

\bigskip 

Albertson and Collins mention the following result (due to 
Shearer)~\cite{kn:albertson} for the graphs $K_e^r(cp(n))$, 
defined inductively by $K_e^r(cp(n))=K_e(K_e^{r-1}(cp(n)))$.
\[\alpha(K_e^r(cp(n))) \leq 3\cdot (2^r)!\] 
Thus, for $r=1$, $\alpha(K_e(cp(n))) \leq 6$. 
However, the following is easily checked. 

\begin{lemma}
\label{bound alpha}
For $n \geq 2$ 
\[\alpha(K_e(cp(n))) =4.\]
\end{lemma}
\begin{proof}
Let $G$ be the complement of 
a matching $\{x_i,y_i\}$, for $i \in \{1,\dots,n\}$. 
Let $K=K_e(G)$. 
Obviously, every pair of edges in the matching induces 
an independent set with four vertices in $K$. 

\medskip 

\noindent
Consider an edge $e=\{x_i,x_j\}$ in $G$. The only 
edges in $G$ that are not adjacent to $e$ in $K$, must have 
endpoints in $y_i$ or in $y_j$. Consider an edge 
$f=\{y_i,y_k\}$ for some $k \notin \{i,j\}$. The only other 
edge incident with $y_i$, 
which is not adjacent in $K$ to $f$ nor to $e$ is 
$\{y_i,x_k\}$. 

\medskip 

\noindent
The only edge incident with $y_j$ which is not adjacent to 
$e$ nor $f$ is $\{y_j,x_i\}$. This proves the lemma. 
\qed\end{proof}

\bigskip 

\begin{definition}
A class of graphs $\mathcal{G}$ is $\chi$-bounded if there 
exists a function $f$ such that for every graph $G \in \mathcal{G}$, 
\[\chi(G) \leq f(\omega(G)).\]
\end{definition}
Dvo\v{r}\'ak and Kr\'al proved that the class of graphs 
with rankwidth at most $k$ is $\chi$-bounded~\cite{kn:dvorak}. 

\bigskip 

We now easily obtain our result. 

\begin{theorem}
The class of edge-clique graphs of cocktail parties has 
unbounded rankwidth. 
\end{theorem}
\begin{proof}
It is easy to see that the rankwidth of any graph 
is at most one more than the rankwidth of its complement~\cite{kn:oum}. 
Assume that there is a constant $k$ such that 
the rankwidth of $K_e(G)$ is at most $k$ whenever $G$ is a 
cocktail party graph. Let 
\[\mathcal{K}=\{\; \overline{K_e(G)} \;|\; G \simeq cp(n), 
\;n \in \mathbb{N}\;\}.\] 
Then the rankwidth of graphs in $\mathcal{K}$ is uniformly bounded 
by $k+1$. By the result of Dvo\v{r}\'ak and Kr\'al, there exists 
a function $f$ such that 
\[\kappa(K_e(G)) \leq f(\alpha(K_e(G)))\] 
for every cocktail party graph $G$. This contradicts 
Lemma~\ref{bound alpha} and Theorem~\ref{gyarfas}. 
\qed\end{proof}

\bigskip 

\begin{remark}
It is easy to see that for cographs $G$, $K_e(G)$ is not 
necessarily perfect. For example, when $G$ is the join of 
$P_3$ and $C_4$ then $K_e(G)$ contains $C_5$ as an induced subgraph. 
\end{remark}

\section{Independent set in edge-clique graphs of cographs}

Notice that, for any graph $G$, the clique number of 
its edge-clique graph satisfies 
\[\omega(K_e(G))=\binom{\omega(G)}{2}.\] 

\bigskip 

For the independent set number $\alpha(K_e(G))$ there is 
no such relation. For example, when $G$ has no triangles 
then $K_e(G)$ is an independent set and the independent 
set problem in triangle-free graphs is NP-complete. 
We write 
\[\alpha^{\prime}(G)=\alpha(K_e(G)).\] 
We say that a subset of edges in a graph $G$ is independent if it 
induces an independent set in $K_e(G)$. In other words, 
a set $A$ of edges in $G$ is independent if no two edges of $A$ 
are contained in a clique of $G$.  

\bigskip 

A graph is trivially perfect if it does not contain $C_4$ nor 
$P_4$ as an induced subgraph. 

\begin{theorem}
If a graph $G$ is connected and trivially perfect then 
$\alpha^{\prime}(G)=\theta_e(G)$. 
\end{theorem}
\begin{proof}
When a graph $G$ is trivially perfect then 
the independence number is equal 
to the number of maximal cliques in $G$. 
Therefore, $\alpha(G)= \theta_e(G)$ and since 
$G$ is connected $\alpha^{\prime}(G) \geq \alpha(G)$. 
\qed\end{proof}
 
\bigskip 

The following lemma shows that the independent set problem 
in $K_e(G)$ can be reduced to the independent set problem in $G$. 

\begin{lemma}
The computation of $\alpha^{\prime}(G)$ 
for arbitrary graphs $G$ is NP-hard.
\end{lemma}
\begin{proof}
Let $G$ be an arbitrary graph.
Construct a graph $H$ as follows. At every edge in $G$ add two
simplicial vertices, both adjacent to the two endvertices of the edge.
Add one extra vertex $x$ adjacent to all vertices of $G$.
Let $H$ be the graph constructed in this way. 

\medskip 

\noindent
Notice that a maximum set of independent edges does not
contain any edge of $G$ since it would be better to replace
such an edge by two edges incident with the two
simplicial vertices at this edge.
Also notice that a set of independent edges incident with $x$
corresponds with an independent set of vertices in $G$.
Hence 
\[\alpha^{\prime}(H) = 2m + \alpha(G),\] 
where
$m$ is the number of edges of $G$. 
\qed\end{proof}

\bigskip 

A cograph is a graph without induced $P_4$. It is well-known that 
a graph is a cograph if and only if every induced subgraph with 
at least two vertices is either a join or a union of two 
smaller cographs. It follows that a cograph $G$ has a decomposition 
tree $(T,f)$ where $T$ is a rooted binary tree and $f$ is a bijection 
from the vertices of $G$ to the leaves of $T$. Each internal node 
of $T$, including the root, is labeled as $\otimes$ or $\oplus$. 
The $\otimes$-node joins the two subgraphs mapped to the left 
and right subtree. The $\oplus$ unions the two subgraphs. 
When $G$ is a cograph then a decomposition tree as described 
above can be obtained in linear time~\cite{kn:corneil}. 
  
\bigskip 

\begin{lemma}
Let $G$ be a cograph. Assume that $G$ is the join of two smaller 
cographs $G_1$ and $G_2$. Then any edge in $G_1$ is adjacent in 
$K_e(G)$ to any edge in $G_2$. 
\end{lemma}
\begin{proof}
Let $e_1$ and $e_2$ be edges in $G_1$ and $G_2$, respectively. 
Then the four endpoints induce a clique in $G$. 
\qed\end{proof}

\bigskip 

For a vertex $x$, let $d^{\prime}(x)$ be the independence 
number of the subgraph of $G$ induced by $N(x)$, that is,  
\begin{equation}
\boxed{d^{\prime}(x)=\alpha(G[N(x)]).}
\end{equation}  

\bigskip 
 
\begin{theorem}
When $G$ is a cograph then 
\begin{equation}
\label{eqn0}
\alpha^{\prime}(G)=\max \;\{\;\sum_{x \in A} d^{\prime}(x)\;|\; 
\text{$A$ is an independent set 
in $G$}\;\}. 
\end{equation}
There exists an $O(n^4)$ algorithm that computes  
the independence number of $K_e(G)$ for cographs $G$. Here $n$ is 
the number of vertices of $G$.  
\end{theorem}
\begin{proof}
Let $G$ be a cograph. 
The algorithm first computes a decomposition tree $(T,f)$ 
for $G$ in linear time. For each node $p$ of $T$ 
let $G_p$ be the subgraph induced by the set of 
vertices that are mapped to leaves in the subtree rooted at $p$.   

\medskip 

\noindent
Notice that the independence 
number of each $G_p$ can be computed in linear time as follows. 

\medskip 

\noindent
Let $p$ be an internal node and let $c_1$ and $c_2$ be the two 
children of $p$. For $i \in \{1,2\}$, 
write $G_i$ instead of $G_{c_i}$. 
Let $p$ be labeled with $\otimes$ and let $G_p$ be the 
join of $G_1$ and $G_2$. 
Then 
\[\alpha(G_p) = \max\;\{\;\alpha(G_1),\;\alpha(G_2)\;\}.\] 
Assume that $p$ is labeled with $\oplus$. Then let $G_p$ be the 
union of $G_1$ and $G_2$. 
In that case 
\[\alpha(G_p)=\alpha(G_1)+\alpha(G_2). \] 

\medskip 

\noindent
Consider two disjoint, nonempty subsets of vertices, $A$ and $B$, 
such that the graph $G[A \cup B]$ is either a join or a union 
of $G[A]$ and $G[B]$.  
Let $\alpha^{\prime}(A,B)$ be the maximal cardinality of 
an independent set of 
edges in $G[A \cup B]$ such that no element has both endpoints 
in $B$. 
Assume that $G[A\cup B]$ is the union of $G[A]$ and $G[B]$. 
Then 
\begin{equation}
\label{eqn1}
\alpha^{\prime}(A,B)=\alpha^{\prime}(G[A]).
\end{equation}
Assume that $G[A \cup B]$ is the join of $G[A]$ and $G[B]$. 
Then consider the following three cases. 
First assume that $G[A]$ is the union of two smaller cographs 
$G[A_1]$ and $G[A_2]$. In that case 
\begin{equation}
\label{eqn2}
\alpha^{\prime}(A,B)=\alpha^{\prime}(A_1,B)+\alpha^{\prime}(A_2,B).
\end{equation}

\medskip 

\noindent
Consider the case where $G[A]$ is the join of two smaller 
cographs $G[A_1]$ and $G[A_2]$. 
In that case 
\begin{equation}
\label{eqn3}
\alpha^{\prime}(A,B)=\max\;\{\;\alpha^{\prime}(A_1,B\cup A_2),\;
\alpha^{\prime}(A_2,B\cup A_1)\;\}.
\end{equation}

\medskip 

\noindent 
Finally, assume that $|A|=1$. In that case 
\begin{equation}
\label{eqn4}
\alpha^{\prime}(A,B)=\alpha(G[B]).
\end{equation}

\medskip 

\noindent
Now, Equation~(\ref{eqn0}) easily follows by induction 
from the recurrences (\ref{eqn1}), (\ref{eqn2}), (\ref{eqn3}) 
and (\ref{eqn4}). 
It is easy to see that this can be computed in $O(n^4)$ time. 
\qed\end{proof}

\bigskip 

\begin{remark}
Notice that Formula~(\ref{eqn0}) confirms 
Lemma~\ref{bound alpha}.
\end{remark}

\bigskip 

\begin{remark}
The independence number of $K_e(G)$ equals $\theta_e(G)$ for 
graphs $G$ that are chordal. For interval graphs the 
edge-clique cover number  
$\theta_e(G)$ equals the number of 
maximal cliques~\cite{kn:scheinerman}.
\end{remark}

\subsection{Distance-hereditary graphs}

In this section we briefly look at the independence number 
of edge-clique graphs of distance-hereditary graphs. 

\bigskip 

A graph $G$ is distance hereditary if the distance 
between any two nonadjacent vertices, 
in any connected induced subgraph of $G$, 
is the same as their distance in the $G$~\cite{kn:howorka}. 
Bandelt and Mulder obtained the following 
characterization of distance-hereditary graphs. 

\begin{lemma}[\cite{kn:bandelt}]
A graph is distance hereditary if and only if every induced subgraph 
has an isolated vertex, a pendant vertex or a twin.
\end{lemma}
The papers~\cite{kn:bandelt} and~\cite{kn:howorka} 
also contain a characterization of 
distance-hereditary graphs in terms of forbidden induced 
subgraphs. 

\bigskip 

\begin{theorem}
Let $G$ be distance hereditary. 
Then $\alpha^{\prime}(G)$ satisfies 
Equation~{\rm (\ref{eqn0})}. This value 
can be computed in polynomial time. 
\end{theorem}
\begin{proof}
Consider an isolated vertex $x$ in $G$. Then $A$ is a maximum 
independent set of edges in $G$ if and only if $A$ is a 
maximum independent set of edges in the graph $G-x$. 
By induction, Equation~(\ref{eqn0}) is valid for $G$. 

\medskip 

\noindent
Let $x$ be a pendant vertex and let $y$ be the 
unique neighbor of $x$ in $G$. Since $\{x,y\}$ is 
not in any triangle, the edge $\{x,y\}$ is in any 
maximal independent set of edges in $G$. 
Therefore, 
\[\alpha^{\prime}(G)=1+\alpha^{\prime}(G-x).\] 
Let $Q$ be an independent set which maximizes 
Equation~(\ref{eqn0}) for $G-x$. If $y \in Q$ 
then $d^{\prime}(y)$ goes up 
by one when adding the vertex $x$. If $y \notin Q$, 
then $Q \cup \{x\}$ is an independent set in $G$ 
and $d^{\prime}(x)=1$. 

\medskip 

\noindent
Let $x$ be a false twin of a vertex $y$ in $G$. 
Let $A$ be a maximum independent set of edges in $G$. 
Let $A(x)$ and $A(y)$ be the  
sets of edges in $A$ that are incident with $x$ and $y$. 
Assume that $|A(x)| \geq |A(y)|$. Let $\Omega(x)$ 
be the set of endvertices in $N(x)$ of edges in $A(x)$.  
then we may replace the set $A(y)$ 
with the set 
\[\{\;\{y,z\}\;|\; z \in \Omega(x)\;\}.\] 
The cardinality of the new set is at least as large as $|A|$. 
Notice that, for any maximal independent set $Q$ in $G$, either 
$\{x,y\} \subseteq Q$ or $\{x,y\} \cap Q = \es$. 
By induction on the number of 
vertices in $G$, Equation~(\ref{eqn0}) is valid. 

\medskip 

\noindent
Let $x$ be a true twin of a vertex $y$ in $G$. 
Let $A$ be a maximum independent set of edges in $G$ and 
let $A(x)$ and $A(y)$ be the sets of  
edges in $A$ that 
are incident with $x$ and $y$. 
If $\{x,y\} \in A$ then $A(x)=A(y)=\{x,y\}$. 

\medskip 

\noindent
Now assume that $\{x,y\} \notin A$. 
Endvertices in $N(x)$ of edges in $A(x)$ and $A(y)$ 
are not adjacent nor do they coincide. Replace $A(x)$ with 
\[\{\;\{x,z\}\;|\; \{x,z\} \in A(x) \quad\text{or}\quad 
\{y,z\} \in A(y) \;\}\] 
and set $A(y)=\es$. Then the new set of edges is independent 
and has the same cardinality as $A$.   

\medskip 

\noindent
Let $Q$ be an independent set in $G$. At most one of 
$x$ and $y$ is in $Q$. 
The validity of Equation~(\ref{eqn0}) 
is easily checked.  
\qed\end{proof}

\section{Graphs without odd wheels}

A wheel $W_n$ is a graph consisting of a cycle $C_n$ and one 
additional vertex adjacent to all vertices in the cycle. 
The universal vertex of $W_n$ is called the hub. It is unique unless 
$W_n=K_4$. The edges incident with the hub are called the 
spokes of the wheel. The cycle is called the rim of the wheel. 
A wheel is odd if the number of vertices in the cycle is odd. 

\bigskip 
 
Lakshmanan, Bujt\'as and Tuza investigate the class of graphs 
without odd wheels in~\cite{kn:lakshmanan}. 
They prove that Tuza's conjecture holds true for 
this class of graphs. 

\bigskip 

Notice that a graph $G$ has no odd wheel if and only if 
every neighborhood in $G$ induces a bipartite graph. It follows 
that $\omega(G) \leq 3$. Obviously, the class of graphs 
without odd wheels is closed under taking subgraphs. 
Notice that, when $G$ has no odd wheel then every neighborhood 
in $K_e(G)$ is either empty or a matching. Furthermore, 
it is easy to see that $K_e(G)$ contains no diamond (every 
edge is in exactly one triangle), no $C_5$ and no  
odd antihole. 

\bigskip 

For graphs $G$ without odd wheels $K_e(G)$ coincides 
with the anti-Gallai graphs introduced by Le~\cite{kn:le}. 
For general anti-Gallai graphs the computation of the clique  
number and chromatic number 
are NP-complete.

\bigskip 

Let us mention that the recognition of anti-Gallai graphs 
of general graphs  
is NP-complete~\cite{kn:anand}. 
The recognition of edge-clique graphs of graphs without odd 
wheels is, as far as we know, open. Let us also mention 
that the edge-clique graphs 
of graphs without odd wheels are clique graphs~\cite{kn:cerioli2}. 
The recognition of 
clique graphs of general graphs is NP-complete~\cite{kn:alcon}.  

\bigskip 

\begin{theorem}
\label{NP_c}
The computation of $\alpha^{\prime}(G)$ is NP-complete 
for graphs $G$ without odd wheels. 
\end{theorem}
\begin{proof}
We reduce 3-SAT to the vertex cover problem in 
edge-clique graphs of graphs without odd wheels. 

\medskip 

\noindent
Let $H \cong L(K_4)$, ie, the complement of $3K_2$.  
Let $S$ be a 3-sun. The graph $H$ is obtained from $S$ by adding 
three edges between pairs of vertices of 
degree two in $S$.%
\footnote{In~\cite[Theorem~14]{kn:lakshmanan2} the authors 
prove that every maximal clique in $K_e(G)$ contains a 
simplicial vertex if and only if $G$ does not contain,  
as an induced subgraph, $K_4$  
nor a 3-sun with 0, 1, 2 or 3 edges connecting the 
vertices of degree two.}
Call the three vertices of degree four in $S$, the `inner 
triangle' of $H$ and call the remaining three vertices of $H$ the 
`outer triangle.' 

\begin{figure}[ht]
  \centering
  \includegraphics[width=0.8\linewidth]{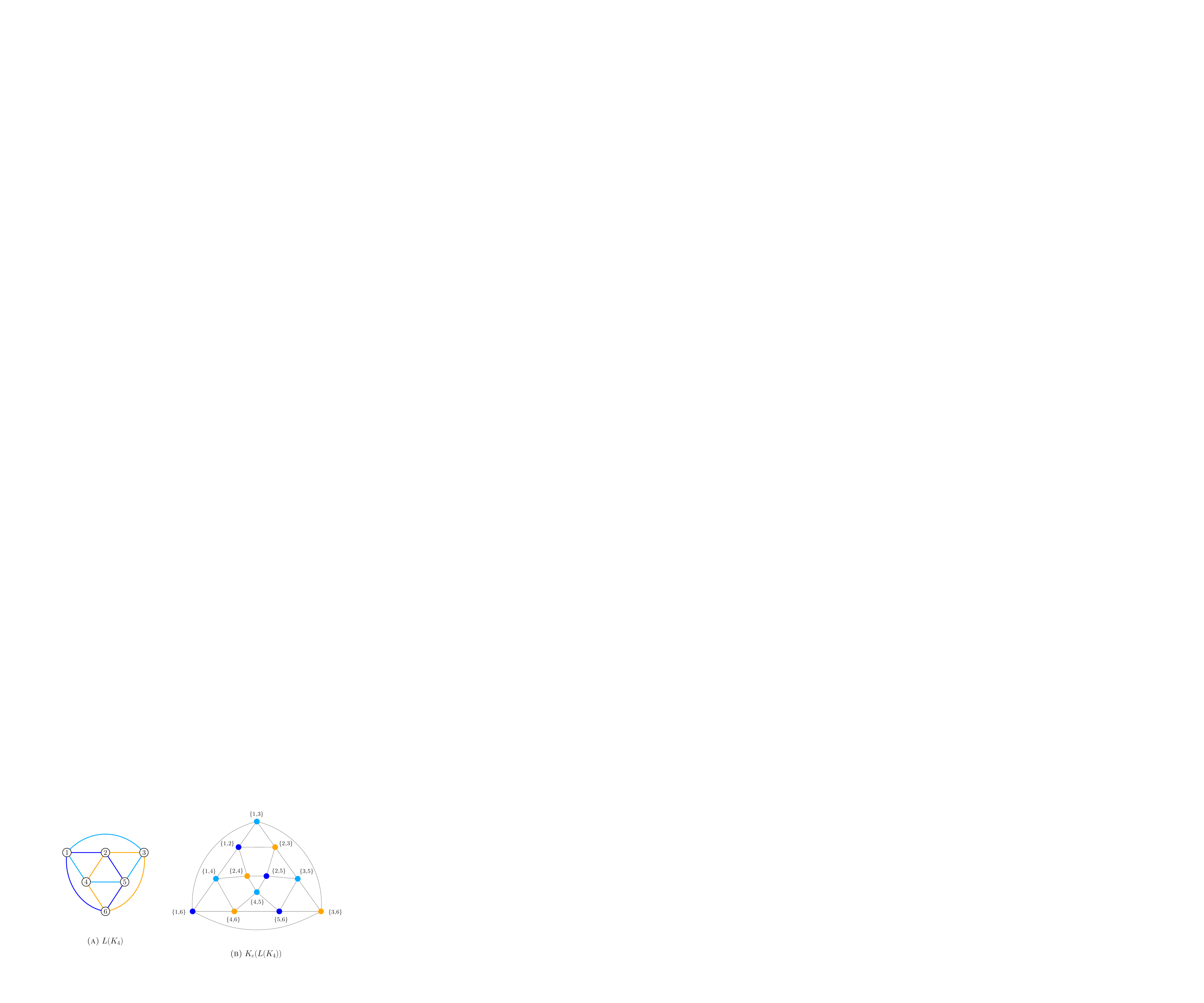}
  \caption{This figure shows $L(K_4)$ and $K_e(L(K_4))$.
  The edges of $L(K_4)$ are colored with three colors
  and so are their corresponding vertices in $K_e(L(K_4))$.
  The three colors indicate the partition into
  three maximum independent sets of edges of $L(K_4)$.
  }
  \label{fig:H_KeH}
\end{figure}

\medskip 

\noindent
Notice that $H$ has 3 maximum independent sets of 
edges. 
Each maximum independent set of edges is an induced $C_4$ 
consisting of one 
edge from the inner triangle, one edge from the outer triangle, 
and two edges between the two triangles. The three independent sets 
partition the edges of $H$. 

\noindent
The six edges of $H$ between the inner and outer triangle form 
a 6-cycle in $K_e(H)$. Let $F$ denote this set of edges in $H$.  

\medskip 

\noindent
For each clause $(x_i \; \vee \; x_j \; \vee \; x_k)$
take one copy of $H$.   
Take an independent set of three edges contained in $F$ 
and label these with $x_i$, $x_j$ and $x_k$. 

\medskip 

\noindent
For each variable $x$ take a triangle. Label one edge 
of the triangle with the literal $x$ and one other edge of 
the triangle with its negation $\Bar{x}$. 

\medskip 

\noindent 
Construct links between variable gadgets and clause gadgets 
as follows. Let $(x_i\;\vee\;x_j\;\vee\;x_k)$ be a 
clause. Identify one endpoint of the edge $x_i$ in the clause 
gadget with an endpoint of the edge labeled $x_i$ in the variable 
gadget. Add an edge between the other two endpoints. 
Construct links for the other two literals in the clause in the 
same manner. 
 
\medskip 

\noindent 
Let $G$ be the graph constructed in this manner. 
Notice that $K_e(G)$ contains some simplicial vertices; 
namely the unlabeled edges in each variable gadget and the 
unlabeled edges in the links. Notice that these simplicial vertices 
can be removed without changing the complexity of the vertex cover 
problem.\footnote{Consider any graph $W$. Assume that $W$ has a 
simplicial vertex $s$. Then the vertex cover number 
satisfies $\kappa(W)=\kappa(W-s)$, unless 
the component that contains $s$ is $\{s\}$ or an edge. 
See eg~\cite[Theorem~2.64]{kn:kloks}.} 
Let $K$ be the graph obtained from $K_e(G)$ by removing these 
simplicial vertices. 

\medskip 

\noindent
Let $L$ be the number of variables and let $M$ be the number of clauses 
in the 3-SAT formula. 
Assume that there is a satisfying assignment. Then choose the vertices in 
$K$  
corresponding to literals that 
are {\sc true} in the vertex cover. 
The variable gadgets need $L$ vertices in the vertex cover.  
Since this assignment is satisfying, we need at most 
$8M$ vertices to cover the remaining edges in the 
clause structures, since the outgoing edge from each literal 
which is {\sc true} is covered. Thus there is a vertex 
cover of $K_e(G)$ with 
$L+8M$ vertices. 

\medskip 

\noindent
Assume that $K_e(G)$  
has a vertex cover with $L+8M$ vertices. At least 
$L$ vertices in $K$ are covering the edges in the variable gadgets. 
The other $8M$ vertices of $K$ are covering the 
edges in the clause gadgets. 
Take the literals of the variable gadgets 
that are in the vertex cover as an assignment 
for the formula. Each clause gadget must have one literal vertex 
of which the outgoing edge is covered. Therefore, the assignment is 
satisfying. 
\qed\end{proof}
 
\section{Concluding remark}

As far as we know, the recognition of 
edge-clique graphs is an open problem. 

\bigskip

Let $K_n^m$ denote the complete multipartite graph with $m$ partite sets 
each having $n$ vertices. Obviously, $K_n^m$ is a cograph with $n\cdot m$ 
vertices. 

\begin{theorem}[\cite{kn:park}]
Assume that 
\[3 \leq m \leq n+1.\] 
Then $\theta_e(K_n^m)=n^2$ if and only if there exists  
a collection of at least $m-2$ pairwise orthogonal Latin squares 
of order $n$. 
\end{theorem}
Notice that, if there exists an edge-clique cover of $K_n^m$ 
with $n^2$ cliques, then these cliques are mutually edge-disjoint. 

\bigskip 

Finding the maximum number of pairwise orthogonal 
Latin squares of order $n$ is a renowned open problem. 
The problem has a wide field 
of applications, eg in combinatorics, designs 
of experiments, group theory and quantum informatics.   

Unless $n$ is a prime power, 
the maximal number of MOLS is known for only a few 
orders. We briefly mention a few results.
Let $f(n)$ denote the 
maximal number of MOLS of order $n$. The well-known `Euler-spoiler'  
shows that $f(n)=1$ only for 
$n=2$ and $n=6$. 
Also, $f(n) \leq n-1$ for all $n>1$, 
and Chowla, Erd\"os and Straus~\cite{kn:chowla} 
show that 
\[\lim_{n \rightarrow \infty} f(n) = \infty.\] 
Define 
\[n_r = \max\;\{\;n\;|\; f(n) < r\;\}.\] 
A lowerbound for the speed at which $f(n)$ grows was obtained 
by Wilson, who showed that $n_r < r^{17}$ when 
$r$ is sufficiently large~\cite{kn:wilson}. 
Better bounds for $n_r$, for some specific values of $r$, were 
obtained by various authors (see eg~\cite{kn:brouwer}). 

See eg ~\cite{kn:feifei} for some recent computational 
attempts to find orthogonal Latin squares. 
The problem seems extremely hard, both from a 
combinatorial and from a computational point of view. 
Despite many efforts, 
the existence of three pairwise orthogonal Latin squares of 
order 10 is, as far as we know, still unclear. 

\begin{conjecture}
The edge-clique cover problem is NP-complete for cographs. 
\end{conjecture}


\begin{thebibliography}{99}

\bibitem{kn:albertson}Albertson,~M. and K.~Collins, 
Duality and perfection for edges in cliques, 
{\em Journal of Combinatorial Theory, Series B\/} {\bf 36} (1984), 
pp.~298--309. 

\bibitem{kn:alcon}Alc\'on,~L., L.~Faria, C.~de~Figueiredo 
and M.~Gutierrez, 
The complexity of clique graph recognition, 
{\em Theoretical Computer Science\/} {\bf 410} (2009), pp.~2072--2083. 

\bibitem{kn:alon}Alon,~N., 
Covering graphs by the maximum number of equivalence relations, 
{\em Combinatorica\/} {\bf 6} (1986), pp.~201--206. 

\bibitem{kn:anand}Anand,~P., H.~Escuadro, R.~Gera, S.~Hartke 
and D.~Stolee, 
On the hardness of recognizing triangular line graphs. 
Manuscript on ArXiV: 1007.1178v1, 2010. 

\bibitem{kn:bandelt}Bandelt,~H. and H.~Mulder, 
Distance-hereditary graphs, 
{\em Journal of Combinatorial Theory, Series B\/} {\bf 41} (1986), 
pp.~182--208. 

\bibitem{kn:berge}Berge,~C., 
{\em Graphs and Hypergraphs\/}, 
North-Holland, Amsterdam, and American Elsevier, New York, 
1973. 

\bibitem{kn:bjorklund}Bj\"orklund,~A., T.~Husfeldt and M.~Koivisto, 
Set partitioning via inclusion-exclusion, 
{\em SIAM Journal on Computing\/} {\bf 39} (2009), pp.~546--563. 

\bibitem{kn:brigham}Brigham,~R.~C. and R.~D.~Dutton, 
Graphs which, with their complements, have certain clique covering numbers, 
{\em Discrete Mathematics\/} {\bf 34} (1981), pp.~1--7. 

\bibitem{kn:brigham2}Brigham,~R.~C. and R.~D.~Dutton, 
On clique covers and independence numbers of graphs, 
{\em Discrete Mathematics\/} {\bf 44} (1983), pp.~139--144. 

\bibitem{kn:brouwer}Brouwer,~A. and G.~Rees, 
More mutually orthogonal Latin squares, 
{\em Discrete Mathematics\/} {\bf 39} (1982), pp.~263--181. 

\bibitem{kn:brugmann}Br\"ugmann,~D., C.~Komusiewics and H.~Moser, 
On generating triangle-free graphs, 
{\em Electronic Notes in Discrete Mathematics\/} {\bf 32} (2009), 
pp.~51--58. 

\bibitem{kn:caen}de~Caen,~D. and N.~Pullman,  
clique coverings of complements of paths and cycles, 
{\em Annals of Discrete Mathematics\/} {\bf 27} (1985), pp.~257--268. 

\bibitem{kn:calamoneri}Calamoneri,~T. and R.~Petreschi, 
Edge-clique graphs and the $\lambda$-coloring problem, 
{\em Journal of the Brazilian Computer Society\/} {\bf 7} (2001), pp.~38--47. 

\bibitem{kn:cerioli2}Cerioli,~M., 
Clique graphs and edge-clique graphs, 
{\em Electronic Notes in Discrete Mathematics\/} {\bf 13} (2003), 
pp.~34--37. 

\bibitem{kn:cerioli}Cerioli,~M., L.~Faria, T.~Ferreira, C.~Martinhon, 
F.~Protti and B.~Reed, 
Partition into cliques for cubic graphs: planar case, complexity 
and approximation, 
{\em Discrete Applied Mathematics\/} {\bf 156} (2008), pp.~2270--2278. 

\bibitem{kn:cerioli4}Cerioli,~M. and J.~Szwarcfiter, 
A characterization of edge clique graphs, 
{\em Ars Combinatorica\/} {\bf 60} (2001), pp.~287--292. 

\bibitem{kn:cerioli3}Cerioli,~M. and J.~Szwarcfiter, 
Edge clique graphs and some classes of chordal graphs, 
{\em Discrete Mathematics\/} {\bf 242} (2002), pp.~31--39. 

\bibitem{kn:choudum}Choudum,~S.~A., K.~Parthasarathy and G.~Ravindra, 
Line-cover number of a graph, 
{\em Indian Nat. Sci. Acad. Proc.\/} {\bf 41} (1975), pp.~289--293. 

\bibitem{kn:chowla}Chowla,~S., P.~Erd\"os and E.~Straus, 
On the maximal number of pairwise orthogonal Latin squares 
of a given order, 
{\em Canadian Journal of Mathematics\/} {\bf 12} (1960), pp.~204--208. 

\bibitem{kn:corneil}Corneil,~D., Y.~Perl and L.~Stewart, 
A linear recognition algorithm for cographs, 
{\em SIAM Journal on Computing\/} {\bf 14} (1985), pp.~926--934. 

\bibitem{kn:cygan}Cygan,~M., M.~Pilipczuk and M.~Pilipczuk, 
Known algorithms for edge clique cover are probably optimal. 
Manuscript on ArXiV: 1203.1754v1, 2012. 

\bibitem{kn:dutton}Dutton,~R.~D. and R.~C.~Brigham, 
A characterization of competition graphs, 
{\em Discrete Applied Mathematics\/} {\bf 6}, 1983, pp.~315--317. 

\bibitem{kn:dvorak}Dvo\v{r}\'ak,~Z. and D.~Kr\'al,
Classes of graphs with small rank decompositions are $\chi$-bounded.
Manuscript on ArXiV: 1107.2161.v1, 2011.

\bibitem{kn:erdos} Erd\"os,~P., W.~Goodman and L.~P\'osa, 
The representation of a graph by set intersections, 
{\em Canadian Journal of Mathematics\/} {\bf 18} (1966), pp.~106--112. 

\bibitem{kn:fleischer}Fleischer,~R. and X.~Wu, 
Edge clique partition of $K_4$-free and planar graphs, 
{\em Proceedings CGGA'10\/}, Springer, LNCS~7033 (2011), pp.~84--95. 

\bibitem{kn:frieze}Frieze,~A. and B.~Reed, 
Covering the edges of a random graph by cliques, 
{\em Combinatorica\/} {\bf 15} (1995), pp.~489--497. 

\bibitem{kn:furedi}F\"uredi,~Z., 
On the double competition number, 
{\em Discrete Applied Mathematics\/} {\bf 82} (1998), pp.~251--255. 

\bibitem{kn:gramm}Gramm,~J., J.~Guo, F.~H\"uffner and R.~Niedermeier, 
Data reduction, exact, and heuristic algorithms for 
clique cover, 
{\em Proceedings ALENEX'06\/}, SIAM (2006), pp.~86--94. 

\bibitem{kn:gregory}Gregory,~D.~A. and N.~J.~Pullman, 
On a clique covering problem of Orlin, 
{\em Discrete Mathematics\/} {\bf 41} (1982), pp.~97--99. 

\bibitem{kn:grotschel}Gr\"otschel,~M., L.~Lov\'asz and 
A.~Schrijver, 
Chapter~9, ``Stable sets in graphs,''~pp.~273--303 
in: {\em Geometric Algorithms and Combinatorial 
Optimization\/}, Springer-Verlag, 1988. 
  
\bibitem{kn:gyarfas}Gy\'arf\'as,~A., 
A simple lowerbound on edge covering by cliques, 
{\em Discrete Mathematics\/} {\bf 85} (1990), pp.~103--104. 

\bibitem{kn:hoede}Hoede,~C. and X.~Li, 
Clique polynomials and independent set polynomials of graphs, 
{\em Discrete Mathematics\/} {\bf 125} (1994), pp.~219--228. 

\bibitem{kn:holyer} Holyer,~I., 
The NP-completeness of some edge-partition problems, 
{\em SIAM Journal on Computing\/} {\bf 4} (1981), pp.~713--717. 

\bibitem{kn:hoover}Hoover,~D.~N., 
Complexity of graph covering problems for graphs of low degree, 
{\em JCMCC\/} {\bf 11} (1992), pp.~187--208. 

\bibitem{kn:howorka}Howorka,~E., 
A characterization of distance-hereditary graphs, 
{\em The Quarterly Journal of Mathematics\/} {\bf 28} (1977), 
pp.~417--420. 

\bibitem{kn:kim}Kim,~Suh-Ryung, 
The competition number and its variants. 
In (Gimbel, Kennedy, Quintas eds.) 
Quo Vadis, Graph Theory, 
{\em Annals of Discrete Mathematics} {\bf 55} (1993), pp.~313--326. 

\bibitem{kn:kloks}Kloks,~T. and Y.~Wang, 
{\em Advanced Graph Algorithms\/}. Manuscript 2012. 
 
\bibitem{kn:kou}Kou,~L.~T., L.~J.~Stockmeyer and C.~K.~Wong, 
Covering edges by cliques with 
regard to keyword conflicts and intersection graphs,
{\em Comm. ACM\/} {\bf 21} (1978), pp.~135--139. 

\bibitem{kn:lovasz}Lov\'asz,~L., 
On coverings of graphs. 
In (P.~Erd\"os and G.~Katona eds.) Proceedings of the 
Colloquium held at Tihany, Hungary (1966), 
Academic Press, New York, 1968, pp.~231--236. 

\bibitem{kn:ma}Ma,~S., W.~D.~Wallis and J.~Wu, 
Clique covering of chordal graphs, 
{\em Utilitas Mathematica\/} {\bf 36} (1989), pp.~151--152. 

\bibitem{kn:ma2}Ma,~S., W.~D.~Wallis and J.~Wu,
The complexity of the clique partition number problem, 
{\em Congressus Numerantium\/} {\bf 66} (1988), pp.~157--164. 

\bibitem{kn:feifei}Ma,~F. and Zhang,~J., 
Finding orthogonal Latin squares using finite 
model searching tools. Manuscript. To appear 
in {\em Science China -- Information Sciences\/} 2011. 

\bibitem{kn:lakshmanan}Lakshmanan,~S., C.~Bujt\'as and Z.~Tuza, 
Small edge sets meeting all triangles of a graph, 
{\em Graphs and Combinatorics\/} {\bf 28} (2012), pp.~381--392. 

\bibitem{kn:lakshmanan2}Lakshmanan,~S. and A.~Vijayakumar, 
Clique irreducibility of some iterative classes of graphs, 
{\em Discussiones Mathematicae Graph Theory\/} {\bf 28} (2008), 
pp.~307--321. 

\bibitem{kn:lampis}Lampis,~M., 
A kernel of order $2k - c \log k$ for vertex cover, 
{\em Information Processing Letters\/} {\bf 111} (2011), pp.~1089--1091. 

\bibitem{kn:le}Le,~V., 
Gallai graphs and anti-Gallai graphs, 
{\em Discrete Mathematics\/} {\bf 159} (1996), pp.~179--189. 

\bibitem{kn:mckay}McKay,~B., P.~Schweitzer and P.~Schweitzer, 
Competition numbers, quasi-line graphs and holes. 
Manuscript on ArXiV: 1110.2933v2, 2012. 

\bibitem{kn:mujuni}Mujuni,~E. and F.~Rosamond, 
Parameterized complexity of the clique partition problem, 
(J.~Harland and P.~Manyem, eds.) {\em Proceedings CATS'08\/}, 
ACS, CRPIT series {\bf 77} (2008), pp.~75--78. 

\bibitem{kn:niedermeier}Niedermeier,~R. and P.~Rossmanith, 
An efficient fixed-parameter algorithm for 
3-hitting set, 
{\em Journal of Discrete Mathematics\/} {\bf 1} (2003), pp.~89--102. 

\bibitem{kn:opsut}Opsut,~R.~J., 
On the computation of the competition number of a graph, 
{\em SIAM J. Alg. Disc. Meth.\/} {\bf 4} (1982), pp.~420--428. 

\bibitem{kn:orlin}Orlin,~J., Contentment in graph theory: covering 
graphs with cliques, 
{\em Proceedings of the Nederlandse Academie van Wetenschappen, 
Amsterdam, Series A\/} {\bf 80} (1977), pp.~406--424. 

\bibitem{kn:oum}Oum,~S., {\em Graphs of bounded rankwidth\/}. 
PhD thesis, Princeton University, 2005. 

\bibitem{kn:park}Park,~B., S.~Kim and Y.~Sano, 
The competition numbers of complete multipartite 
graphs and mutually orthogonal Latin squares, 
{\em Discrete Mathematics\/} {\bf 309} (2009), pp.~6464--6469. 

\bibitem{kn:pullman}Pullman,~N.~J., 
Clique covering of graphs IV. Algorithms, 
{\em SIAM Journal on Computing\/} {\bf 13} (1984), pp.~57--75. 

\bibitem{kn:pullman2}Pullman,~N.~J., 
Clique coverings of graphs--A survey, 
{\em Proceedings of the Xth Australian conference 
on combinatorial mathematics\/}, Adelaide 1982. 

\bibitem{kn:roberts}Roberts,~F.~S., 
Applications of edge coverings by cliques, 
{\em Discrete Applied Mathematics\/} {\bf 10} (1985), pp.~93--109. 

\bibitem{kn:roberts2}Roberts,~F.~S. and J.~E.~Steif, 
A characterization of competition graphs of arbitrary digraphs, 
{\em Discrete Applied Mathematics\/} {\bf 6} (1986), pp.~323--326. 

\bibitem{kn:sano}Sano,~Y., 
A generalization of Opsut's lower bounds for the competition 
number of a graph. 
Manuscript on ArXiV 1205.4322v1, 2012. 

\bibitem{kn:scheinerman}Scheinerman,~E. and A.~Trenk, 
On the fractional intersection number of a graph, 
{\em Graphs and Combinatorics\/} {\bf 15} (1999), pp.~341--351. 

\bibitem{kn:shaohan}Shaohan,~M., W.~Wallis and W.~Lin, 
The complexity of the clique partition number problem, 
{\em Congressus Numerantium\/} {\bf 67} (1988), pp.~56--66. 

\bibitem{kn:wilson}Wilson,~R., 
Concerning the number of mutually orthogonal Latin squares, 
{\em Discrete Mathematics\/} {\bf 9} (1974), pp.~181--198. 

\bibitem{kn:wu}Wu,~Y. and J.~Lu, 
Dimension-2 poset competition numbers and dimension-2 poset 
double competition numbers, 
{\em Discrete Applied Mathematics\/} {\bf 158} (2010), pp.~706--717. 
  
\end{thebibliography}
\end{document}